\documentclass[A4paper,reqno,12pt]{amsart}
\usepackage{amssymb}
\usepackage{amsmath, mathrsfs,bm}
\usepackage{enumerate,color,comment}
\usepackage{hyperref}

\makeatletter
\@namedef{subjclassname@2020}{%
  \textup{2020} Mathematics Subject Classification}
\makeatother

\theoremstyle{defi}
\newtheorem{thrm}{Theorem}[section]
\newtheorem{cor}[thrm]{Corollary}
\newtheorem{lem}[thrm]{Lemma}

\newtheorem{prop}[thrm]{Proposition}

\numberwithin{equation}{section}

\DeclareMathOperator{\rad}{rad}
\DeclareMathOperator{\soc}{soc}
\DeclareMathOperator{\tr}{tr}

\DeclareMathOperator{\spn}{span}

\newcommand{\R}{\mathrm{rank}^{\sigma}}

\begin{document}

\title[Banach algebra mappings preserving the invertibility of linear pencils]{Banach algebra mappings preserving the invertibility of linear pencils}

\author{Francois Schulz}
\address{Department of Mathematics and Applied Mathematics, Faculty of Science, University of Johannesburg,
P.O. Box 524, Auckland Park, 2006, South Africa}
\email{francoiss@uj.ac.za}


\date{\today}

\begin{abstract}
Let $A$ and $B$ be complex unital Banach algebras, and let $\varphi, \psi: A \to B$ be surjective mappings. If $A$ is semisimple with an essential socle and $\varphi$ and $\psi$ preserves the invertibility of linear pencils in both directions, that is, for any $x, y \in A$ and $\lambda \in \mathbb{C}$, $\lambda x+y$ is invertible in $A$ if and only if $\lambda \varphi(x) + \psi(y)$ is invertible in $B$, then we show that there exists an invertible element $u$ in $B$ and a
Jordan isomorphism $J: A \to B$ such that $\varphi(x) = \psi(x) = uJ(x)$ for all $x \in A$.
\end{abstract}

\maketitle

\section{Introduction}\label{sect:intro}

\noindent

Throughout $A$ and $B$ shall denote complex unital Banach algebras, where we will always use $\mathbf{1}$ to denote the identity of an algebra under consideration. The group of invertible elements of a Banach algebra $A$ shall be denoted by $G(A)$. For any element $x$ in a unital Banach algebra $A$, recall that its spectrum is given by $$\sigma(x) = \{\lambda \in \mathbb{C}: \lambda \mathbf{1} - x \notin G(A)\}.$$ 
Moreover, we use $\sigma'(x) = \sigma(x)\setminus\{0\}$ and $\rho(x) = \sup_{\lambda \in \sigma(x)}|\lambda|$ to denote its nonzero spectrum and spectral radius, respectively. For the Jacobson radical of a Banach algebra $A$ we shall write $\rad(A)$. We say that $A$ is semisimple if $\rad(A) = \{0\}$. If $A$ is semisimple, then it has a smallest two-sided ideal containing all minimal left and right ideals, called the socle of $A$ which we denote by $\soc(A)$. Lastly we mention that if $A$ is semisimple, then $\soc(A)$ is essential if and only if its left annihilator is trivial.

The study of linear mappings preserving the invertibility or eigenvalues of matrices can be traced back to Dieudonn\'{e} in \cite{dieudonne1948generalisation}. Let $M_n(\mathbb{C})$ denote the Banach algebra of all complex $n \times n$ matrices. Building on the work of Dieudonn\'{e}, Marcus and Purves showed in \cite{marcus1959linear} that if a linear map $\varphi: M_n(\mathbb{C}) \to M_n(\mathbb{C})$ preserves invertibility (that is, $x \in G(M_n(\mathbb{C})) \implies \varphi(x) \in G(M_n(\mathbb{C}))$), then there exist $u, v \in G(M_n(\mathbb{C}))$ such that
\begin{equation}
	\varphi(x) = uxv \mbox{ for all } x \in M_n(\mathbb{C})\;\; \mbox{ or }\;\; \varphi(x) = ux^tv \mbox{ for all } x \in M_n(\mathbb{C}),
	\label{eqform}
\end{equation}
where $x^t$ denotes the transpose of $x$. Of course, if the mapping $\varphi$ above is unital, that is, $\varphi(\mathbf{1}) = \mathbf{1}$, then $u = v^{-1}$. Notice also that one does not need to assume here that $\varphi$ is surjective (since injectivity can be established from the hypotheses). Motivated by the preceding result of Marcus and Purves, and by Kaplansky's famous problem, Sourour classified in \cite{Sourour1996} the form of any bijective linear invertibility preserving map between $\mathcal{L}(X)$ and $\mathcal{L}(Y)$, where $\mathcal{L}(X)$ and $\mathcal{L}(Y)$ are the Banach algebras of all bounded linear operators acting on the Banach spaces $X$ and $Y$, respectively. This was further generalized by Bre\v{s}ar, Fo\v{s}ner, and \v{S}emrl in \cite{bfs2003}, where it is shown that if $\varphi: A \to B$ is a unital bijective linear mapping that preserves invertibility, with $A$ and $B$ semisimple and $\soc(A)$ essential, it follows that $\varphi$ is a Jordan isomorphism; that is, a linear bijective map satisfying $\varphi(x^2) = \varphi(x)^2$ for all $x \in A$. In more recent years there has been much interest in general preserver problems, where the condition of linearity is removed from the hypotheses and the preserving condition is replaced by something which can potentially be used to recover linearity (or semi-linearity) in the conclusion. For example, it is easy to see that a linear spectrum preserving map $\varphi:A \to B$ has the property that $\sigma(\lambda x + y) = \sigma(\lambda \varphi(x) + \varphi(y))$ for any $x, y \in A$ and $\lambda \in \mathbb{C}$. However, if $\varphi$ is not assumed to be linear, can one recover linearity from the preserving property alone? For recent work in this direction see \cite{Askes2022} and \cite{Benjamin2023}.

In this note we shall study pairs of surjective mappings $\varphi, \psi: A \to B$ preserving the invertibility of linear pencils in both directions, that is, which satisfy the property that
\begin{equation}
	\lambda x+y \in G(A) \iff \lambda \varphi(x) + \psi(y) \in G(B) \mbox{ for all }x, y \in A \mbox{ and }\lambda \in \mathbb{C}.
	\label{eqcond}
\end{equation}
Our work is motivated by a recent paper of Costara, \cite{COSTARA2020216}, which generalizes a result on determinant preserving mappings by Dolinar and \v{S}emrl in \cite{DOLINAR2002189}, and deals with our situation in the setting of complex matrices, i.e. the case where $\varphi, \psi$ satisfy \eqref{eqcond} and $A = B = M_n(\mathbb{C})$. Remarkably, if only \textit{one} of the mappings in this case is either surjective or continuous, then \cite[Theorem 1]{COSTARA2020216} says that $\varphi = \psi$ and $\varphi$ takes one of the forms in \eqref{eqform}. Our aim here is to show that if a pair of surjective mappings $\varphi, \psi: A \to B$ satisfy \eqref{eqcond}, where  $A$ is semisimple with $\soc(A)$ essential, then $\varphi = \psi$ and $\varphi: x \mapsto uJ(x)$, where $u \in G(B)$ and $J: A \to B$ is a Jordan isomorphism. Of course, since our consideration includes infinite-dimensional Banach algebras (in particular, $\mathcal{L}(X)$ with $X$ of any dimension), and since we do not place any restrictions on $B$, our assumption that \textit{both} mappings are surjective seems reasonable.

\section{Variation of elements in terms of invertibility}

We firstly prove a few preliminary results concerning the radical and the variation of elements in terms of invertibility in a semisimple Banach algebra. These are similar to some of the results in \cite{COSTARA2020216} and \cite{Havlicek2006}, which were obtained for complex matrices and bounded linear operators on a Hilbert space. However, our proofs use contemporary techniques from spectral theory in abstract Banach algebras and, in this generality, some of our results seem to be new. The first result is probably well-known, but it will be useful in sequel.

\begin{prop}\label{prop1}
	For any $x \in A$ it follows that
	$$x \in \rad(A) \iff x + y \in G(A) \mbox{ for all } y \in G(A).$$
\end{prop} 

\begin{proof}
	The forward implication follows from \cite[Theorem 3.1.3]{aupetit1991primer}, which also implies that
	$$x \in \rad(A) \iff \sigma(yx) = \{0\} \mbox{ for all }y \in A.$$
	Conversely, if $y + x \in G(A)$ for all $y \in G(A)$, then $\sigma(yx) = \{0\}$ for all $y \in G(A)$. Let $w \in A$ be given. Then $\sigma((\lambda\mathbf{1}+w)x) = \{0\}$  for all $\lambda \in \mathbb{C}$ with $|\lambda| > \rho(w)$. If we now take $f: \mathbb{C} \to A$ to be the entire function defined by $f(\lambda) = (\lambda\mathbf{1}+w)x$ for each $\lambda \in \mathbb{C}$, then
	by Aupetit's Scarcity Theorem \cite[Theorem 3.4.25]{aupetit1991primer} and \cite[Corollary 3.4.18]{aupetit1991primer} we infer that $\sigma(f(\lambda)) = \{0\}$ for all $\lambda \in \mathbb{C}$. Taking $\lambda = 0$ allows us to conclude that $\sigma(wx) = \{0\}$ for all $w \in A$. Hence, $x \in \rad(A)$.
\end{proof}

The next result is reminiscent of \cite[Lemma 1]{COSTARA2020216} for matrices.

\begin{prop}\label{prop2}
	Suppose that $A$ is semisimple. For any $a, b \in A$, if 
	$$x+a \in G(A) \iff x + b \in G(A) \mbox{ for all }x \in A,$$
	then $a = b$.
\end{prop}

\begin{proof}
	By replacing $x$ by $-\lambda\mathbf{1}+x$, the hypothesis implies that
	$$\sigma(x+a) = \sigma(x+b) \mbox{ for all }x \in A.$$
	Thus, since $A$ is semisimple, we can apply either \cite[Theorem 2.2]{BraatBrit} or \cite[Theorem 2.1]{SCHULZ20161626} to infer that $a = b$.
\end{proof}

It turns out that it is enough to restrict $x$ in Proposition~\ref{prop2} to the set $G(A)$. For operators, a similar result can be found in \cite[Lemma 3.2]{Havlicek2006}. To prove the result for semisimple Banach algebras in general, we closely follow an argument which appears in the proof of \cite[Theorem 2.6]{BraatBrit}.

\begin{thrm}\label{thrm1}
	Let $A$ be semisimple with $a, b \in A$. Assume that for every $x \in G(A)$ we have that
	$x+a \in G(A) \iff x + b \in G(A)$.
	Then $a = b$.
\end{thrm}

\begin{proof}
	Let $x \in A$ and $\lambda \in \mathbb{C}$ with $|\lambda|> \rho(x)$ be given. Then $-\lambda\mathbf{1}+x \in G(A)$. Thus, the hypothesis implies that
	$$\lambda \mathbf{1} - (x+a) \in G(A) \iff \lambda \mathbf{1} - (x+b) \in G(A).$$
	Consequently, for any $x \in A$ and $\lambda \in \mathbb{C}$, we have that
	\begin{equation}
		|\lambda| > \rho(x) \implies \left[\lambda \in \sigma(x+a) \iff \lambda \in \sigma(x+b) \right].
		\label{E1}
	\end{equation}
	Take any quasinilpotent $q \in A$ (i.e. $q \in A$ with $\rho(q) = 0$). We claim that
	\begin{equation}
		\max\{\rho(n(b-a)+q), \rho(nb-(n-1)a + q)\} \leq \|b\|+\|q\|
		\label{E2}
	\end{equation}
for all $n \in \mathbb{N}$. We proceed by induction on $n$. Certainly,
$$\rho(b+q) \leq \|b+q\| \leq \|b\|+\|q\|.$$
Moreover, by \eqref{E1} with $x = -b-q$ we have
$$ \rho(b-a+q) = \rho(a-b-q) \leq \rho(-b-q) = \rho(b+q),$$
which establishes the base case. Assume now that \eqref{E2} holds for $n = k$. Taking $x = k(b-a)+q$ in \eqref{E1} then yields
$$\rho((k+1)b-ka+q) \leq \|b\|+\|q\|.$$
Moreover, taking $x = -\left((k+1)b-ka+q\right)$ in \eqref{E1} now gives
 $$ \rho((k+1)(b-a)+q) = \rho((k+1)(a-b)-q) \leq \|b\|+\|q\|.$$
 By induction, this then proves our claim that \eqref{E2} holds for all $n \in \mathbb{N}$. For any $\lambda \in \mathbb{C}$, the hypothesis allows us to replace $a$ and $b$ in \eqref{E1} by $\lambda a$ and $\lambda b$. Moreover, recall that $q \in A$ was an arbitrary quasinilpotent element, and notice that $nq$ is quasinilpotent for all $n \in \mathbb{N}$. Consequently, we conclude that
 \begin{equation}
 	\rho(n\lambda(b-a)+nq) \leq \|\lambda b\| + \|nq\| \mbox{ for all }n \in \mathbb{N} \mbox{ and }\lambda \in \mathbb{C}.
 	\label{E3}
 \end{equation}
If we divide \eqref{E3} throughout by $n$ and then let $n \rightarrow \infty$, we obtain that
\begin{equation}
	\rho(\lambda(b-a)+q) \leq \|q\| \mbox{ for all }\lambda \in \mathbb{C}.
	\label{E4} 
\end{equation}
By Vesentini's Theorem \cite[Theorem 3.4.7]{aupetit1991primer} it follows that the map $\lambda \mapsto \rho(\lambda(b-a)+q)$ is subharmonic on $\mathbb{C}$. Since this map is bounded by \eqref{E4}, it is constant by Liouville's Spectral Theorem \cite[Theorem 3.4.14]{aupetit1991primer}. Taking $\lambda = 0$ and $\lambda = 1$ then yields 
$$\rho((b-a)+q) = \rho(q) = 0.$$ 
Since the above holds for any quasinilpotent element $q \in A$, we can now use Zem\'{a}nek's characterization of the radical \cite[Theorem 5.3.1]{aupetit1991primer} to conclude that $b-a \in \rad(A) = \{0\}$.
\end{proof}

If we assume in addition that $a, b \in G(A)$ and that $\soc(A)$ is an essential ideal, then it is possible to restrict $x$ in Proposition~\ref{prop2} to the set of (spectrally) rank one elements of $A$, which we will denote by $\mathscr{F}_1(A)$. Firstly, however, we digress to give a brief description of the rank and trace in the setting of a semisimple Banach algebra.

The \emph{spectral rank} of an element $ a\in A $, is given by
$$\R(a)=\sup_{x\in A}\#\sigma'(xa)\leq \infty,$$
where $\# K$ denotes the number of distinct elements in the set $K$. From the definition it is clear that $\R(ua) = \R(a)$ for each $u \in G(A)$. Moreover, observe that $x \in \rad(A)$ if and only if $\R(x) = 0$. We also note that this definition of rank coincides with the classical operator rank when $ A=\mathcal{L}(X)$. In particular, if $A$ is semisimple, Aupetit and Mouton have shown that $\soc(A)$ is equal to the set of all finite-rank elements of $A$; in fact, $\soc(A)=\spn(\mathscr{F}_1(A))$ (see \cite[Corollary 2.9]{aupetitmoutontrace} and its proof). If $a\in \soc(A)$, then {\cite[Theorem 2.2]{aupetitmoutontrace}} says that the set
$$E(a):=\left\{x\in A: \#\sigma'(xa)=\R(a)\right\}$$
is dense and open in $A$. We say that $a \in \soc(A)$ is a \emph{maximal finite-rank} element if and only if $\mathbf{1}\in E(a)$. If $a$ is a maximal finite-rank element, then by \cite[Theorem 2.8]{aupetitmoutontrace} it can be expressed as a linear combination of mutually orthogonal rank one idempotents, where the scalar coefficients are distinct and exhaust all the values in $\sigma'(a)$. The notion of the \emph{trace} in the general setting of a Banach algebra was firstly studied together with the rank in \cite{aupetitmoutontrace}. It turns out that the trace will be a useful tool in our work. For $a \in \soc(A)$, we define the trace of $a$ by
$$\tr(a) = \sum_{\lambda \in \sigma (a)} \lambda m\left(\lambda, a\right),$$
where $m(\lambda,a)$ is the \emph{multiplicity} of $a$ at $\lambda$. For more details on the multiplicity, see \cite[Theorem 2.4]{aupetitmoutontrace} and the accompanying discussion. Importantly, we mention here that if $a \in \mathscr{F}_1(A)$, then
\begin{equation}
\tr(a) = \left\{\begin{array}{ll} 0 & \mbox{if }\sigma(a) = \{0\}, \\
	\lambda & \mbox{if } \sigma'(a) = \{\lambda\}.
\end{array} \right.
\label{eqa}
\end{equation}
The (generalized) trace retains many of its classical properties; for example, $ a\mapsto \tr(a) $ is a linear functional on $ \soc(A) $ (\cite[Theorem 3.3]{aupetitmoutontrace} and \cite[Lemma 2.1]{tracesocleident}) and it satisfies the well-known \emph{trace property}, that is, $\tr(ab)=\tr(ba)$ for any $a\in \soc(A) $ and $b \in A$ (\cite[Theorem 2.4]{bbs151}). From the linearity of the trace and the preceding discussion, if $a \in \soc(A)$ is a maximal finite-rank element, then
\begin{equation}
\tr(a) = \sum_{\lambda \in \sigma'(a)} \lambda.
\label{eqb}
\end{equation}
Although the trace is not continuous in general, it does follow from \cite[Theorem 3.3]{aupetitmoutontrace} that for any fixed $n \in \mathbb{N}$, the trace is continuous on the set of elements with rank at most $n$. Finally, we highlight the following important property of the trace which will be useful in our work:
\begin{thrm}[{\cite[Corollary 3.6]{aupetitmoutontrace}}]\label{trprop}
	For any $x \in A$, if $\tr(xa) = 0$ for each $a \in \soc(A)$, then $x\soc(A) = \{0\}$. Moreover, if $x \in \soc(A)$, then $\tr(xa) = 0$ for each $a \in \soc(A)$ implies that $x = 0$.
\end{thrm}
We are now in a position to extend \cite[Lemma 3.3]{Havlicek2006} for Hilbert space operators to semisimple Banach algebras with essential socles.

\begin{prop}\label{prop3}
	Let $A$ be semisimple and suppose that $\soc(A)$ is an essential ideal. If $a, b \in G(A)$ and
	$$x+a \in G(A) \iff x+b \in G(A) \mbox{ for all } x \in \mathscr{F}_1(A),$$
	then $a = b$.
\end{prop}

\begin{proof}
	For any $x \in \mathscr{F}_1(A)$ and $\alpha \in \mathbb{C}$, the assumption gives
	$$\mathbf{1}-\alpha a^{-1}x \in G(A) \iff \mathbf{1}-\alpha b^{-1}x \in G(A).$$
	Hence, $\sigma'(a^{-1}x) = \sigma'(b^{-1}x)$ for all $x \in \mathscr{F}_1(A)$. Since $x \in \mathscr{F}_1(A)$, it follows from \eqref{eqa} that
	$$\tr(a^{-1}x) = \tr(b^{-1}x) \mbox{ for each } x \in \mathscr{F}_1(A).$$
	By the linearity of the trace and the fact that $\soc(A)=\spn(\mathscr{F}_1(A))$, we conclude that
	$$\tr((a^{-1}-b^{-1})y) = 0 \mbox{ for all } y \in \soc(A).$$
	Hence, by Theorem~\ref{trprop}, we have $(a^{-1}-b^{-1})\soc(A) = \{0\}$. Since $\soc(A)$ is essential, we therefore obtain that $a = b$ as desired.
\end{proof}

\section{Classification of mappings preserving the invertibility of linear pencils}

In this section we prove our main result which we state next.

\begin{thrm}\label{main}
	Let $\varphi, \psi: A \to B$ be two surjective mappings, where $A$ is semisimple and $\soc(A)$ is essential. If $\varphi$ and $\psi$ satisfy \eqref{eqcond}, that is,
	\begin{equation*}
		\lambda x+y \in G(A) \iff \lambda \varphi(x) + \psi(y) \in G(B) \mbox{ for all }x, y \in A \mbox{ and }\lambda \in \mathbb{C},
	\end{equation*}
	then there exist some $u \in G(B)$ and a Jordan isomorphism $J: A \to B$ such that
	$$\varphi(x) = \psi(x) = uJ(x) \mbox{ for all }x \in A.$$ 
\end{thrm}

Before we proceed to give a proof of our main theorem, we firstly establish a few lemmas which reveal some useful properties of the preservers under consideration.

\begin{lem}\label{inject}
	Let $\varphi, \psi: A \to B$ be two mappings, where $A$ is semisimple. If $\varphi$ and $\psi$ satisfy \eqref{eqcond}, then $\varphi$ and $\psi$ are injective.
\end{lem}

\begin{proof}
	Let $a, b \in A$ and suppose that $\varphi(a) = \varphi(b)$. Then
	$$\varphi(a)+ \psi(x) \in G(B) \iff \varphi(b) + \psi(x) \in G(B) \mbox{ for all } x \in A.$$
	Consequently, by \eqref{eqcond}, we deduce that
	$$a+ x \in G(A) \iff b + x \in G(A) \mbox{ for all } x \in A.$$
	Since $A$ is semisimple, Proposition~\ref{prop2} yields that $a = b$. Similarly one shows that $\psi(a) = \psi(b) \implies a = b$.
\end{proof}

\begin{lem}\label{invert}
	Let $\varphi, \psi: A \to B$ be two surjective mappings, where $A$ is semisimple. If $\varphi$ and $\psi$ satisfy \eqref{eqcond}, then:
	\begin{itemize}
		\item[\textnormal{(i)}] $\psi(G(A)) = G(B)$.
		\item[\textnormal{(ii)}] 
		$B$ is semisimple.
		\item[\textnormal{(iii)}] 
		$\varphi(0) = \psi(0) = 0$.
		\item[\textnormal{(iv)}]
		$\varphi(G(A)) = G(B)$. 
	\end{itemize}
\end{lem}

\begin{proof}
	\begin{itemize}
		\item[\textnormal{(i)}] This follows immediately from condition \eqref{eqcond} with $\lambda = 0$ and the surjectivity of $\psi$.
		 \item[\textnormal{(ii)}]
		 Let $q \in \rad(B)$ be given. Since $\varphi$ is surjective, it follows that $q = \varphi(x)$ for some $x \in A$. From Proposition~\ref{prop1} and part \textnormal{(i)} above it follows that
		 $$\varphi(x) + \psi(y) \in G(B) \mbox{ for each } y \in G(A).$$
		 Consequently, by condition \eqref{eqcond}, we have
		 $$x + y \in G(A) \mbox{ for each } y \in A.$$
		 Thus, if we apply Proposition~\ref{prop1} once again, we can now conclude that $x \in \rad(A) = \{0\}$. This shows that $\rad(B) = \{\varphi(0)\}$, and so, $B$ is semisimple.
		 \item[\textnormal{(iii)}]
		 From $0 \in \rad(B) = \{\varphi(0)\}$, it readily follows that $\varphi(0) = 0$. To see that $\psi(0) = 0$, notice first that $G(B) \subseteq \varphi(G(A))$. Indeed, if $y \in G(B)$, then by the surjectivity of $\varphi$ it follows that $y =\varphi(a)$ for some $a \in A$. If we assume that $a \notin G(A)$, then by condition \eqref{eqcond} we would have for all $\lambda \neq 0$ that
		 \begin{eqnarray*}
		 	\lambda a + 0 \notin G(A) & \implies & \lambda \varphi(a) + \psi(0) \notin G(B) \\ & \implies & \lambda \mathbf{1} + \varphi(a)^{-1}\psi(0) \notin G(B).
		 \end{eqnarray*}
	 However, then $\sigma(\varphi(a)^{-1}\psi(0))$ would be unbounded which is impossible. From this contradiction we conclude that $a \in G(A)$, and so, $G(B) \subseteq \varphi(G(A))$. With this observation, we can now proceed to show that $\psi(0) \in \rad(B)$. For any $y = \varphi(a) \in G(B)$, we have that $a \in G(A)$, and so, by condition \eqref{eqcond}
	 $$a + 0 \in G(A) \implies \varphi(a) + \psi(0) \in G(B).$$
	 Hence, $y + \psi(0) \in G(B)$ for all $y \in G(B)$. By Proposition~\ref{prop1} we therefore have $\psi(0) \in \rad(B) = \{0\}$ as desired.
		 \item[\textnormal{(iv)}] This follows easily from $\psi(0) = 0$, condition \eqref{eqcond}, and the surjectivity of $\varphi$. \qedhere
	\end{itemize}
\end{proof}

Our key strategy is to establish the linearity of $\varphi$ and $\psi$. This will eventually be achieved by showing that $\varphi^{-1}$ and $\psi^{-1}$ are linear under the hypotheses in Theorem~\ref{main}. However, for the time being, we focus on $\varphi$ and $\psi$. It is straightforward to show, using condition \eqref{eqcond}, that we can factor out scalars.   

\begin{lem}\label{hom}
	Let $\varphi, \psi: A \to B$ be two surjective mappings, where $A$ is semisimple. If $\varphi$ and $\psi$ satisfy \eqref{eqcond}, then $\varphi$ and $\psi$ are homogeneous.
\end{lem}

\begin{proof}
From condition \eqref{eqcond} it readily follows that
$$\varphi(\lambda x) + \psi(y) \in G(B) \iff \lambda x + y \in G(A) \iff \lambda \varphi(x) + \psi(y) \in G(B)$$
for all $x, y \in A$ and $\lambda \in \mathbb{C}$. Thus, since $\psi$ is surjective and $B$ is semisimple by Lemma~\ref{invert}, we conclude from Proposition~\ref{prop2} that $\varphi(\lambda x) = \lambda \varphi(x)$ for all $x \in A$ and $\lambda \in \mathbb{C}$. From Lemma~\ref{invert} we recall that $\psi(0) = 0$.  Thus, if we combine this with condition \eqref{eqcond} and use the fact that $G(A)$ is a multiplicative group, we can conclude that
\begin{equation}
	 x+\lambda y \in G(A) \iff  \varphi(x) + \lambda\psi(y) \in G(B) \mbox{ for all }x, y \in A \mbox{ and }\lambda \in \mathbb{C}.
	\label{eqcond2}
\end{equation}
Hence, since $\varphi$ is surjective and $B$ is semisimple by Lemma~\ref{invert}, the same argument as before also yields that $\psi(\lambda y) = \lambda \psi(y)$ for all $y \in A$ and $\lambda \in \mathbb{C}$.
\end{proof}

The next step is to show that the restrictions of the mappings $\varphi$ and $\psi$ to $\soc(A)$ are additive (linear) mappings onto $\soc(B)$. For this we will require that $\varphi$ and $\psi$ both preserve rank. The latter readily follows from an equivalent description of rank which appears in \cite{Askes2022} (compare also \cite[Theorem 2.1]{aupetitmoutontrace}).

\begin{thrm}[{\cite[Theorem 2.2]{Askes2022}}]\label{addthrm}
	Suppose that $A$ is semisimple. Let $a \in A$, let $m \in \mathbb{N}$, and let $K$ be any subset of $\mathbb{C}$ with at least $m+1$ nonzero elements. Then the following are equivalent:
	\begin{itemize}
		\item[\textnormal{(a)}]
		$\R(a) = m$.
		\item[\textnormal{(b)}]
		$\sup_{y \in G(A)} \#\left\{t \in K: y + ta \notin G(A)\right\} = m$.
	\end{itemize}
\end{thrm}

In light of Theorem~\ref{addthrm} and Lemma~\ref{invert}, we now deduce the following:

\begin{lem}\label{rank}
	Let $\varphi, \psi: A \to B$ be two surjective mappings, where $A$ is semisimple. If $\varphi$ and $\psi$ satisfy \eqref{eqcond}, then $\R(a)=\R(\varphi(a))=\R(\psi(a))$ for all $a \in A$.
\end{lem}

\begin{lem}\label{soc}
	Let $\varphi, \psi: A \to B$ be two surjective mappings, where $A$ is semisimple. If $\varphi$ and $\psi$ satisfy \eqref{eqcond}, then the restrictions
	$\varphi: \soc(A) \to \soc(B)$ and $\psi: \soc(A) \to \soc(B)$ are well-defined linear bijective mappings.
\end{lem}

\begin{proof}
From Lemma~\ref{hom} it follows that $B$ is semisimple. Moreover, Lemma~\ref{rank} says that $\varphi$ and $\psi$ both preserve rank. Thus, since $\varphi$ and $\psi$ are bijections by Lemma~\ref{inject} and the hypotheses, we can therefore conclude that the restrictions are well-defined bijective mappings. According to Lemma~\ref{hom} we also already have the homogeneity of $\varphi$ and $\psi$. It therefore remains to show that $\varphi$ and $\psi$ are additive on $\soc(A)$. Since $\soc(A) = \spn(\mathscr{F}_1(A))$, it suffices to show that 
$$\psi(x_1 + \cdots + x_n) = \psi(x_1)+ \cdots + \psi(x_n) \mbox{ for any } x_1, \ldots, x_n \in \mathscr{F}_1(A),$$
and likewise for $\varphi$. To this end, fix any $x_1, \ldots, x_n \in \mathscr{F}_1(A)$ and let $x = x_1 + \cdots + x_n$.  Next let $u \in E(\psi(x))\cap G(B)$ be arbitrary. From Lemma~\ref{invert} we have that $\varphi(G(A)) = G(B)$, and so, $u = \varphi(y)^{-1}$ for some $y \in G(A)$. Using condition \eqref{eqcond} we now observe that
\begin{eqnarray*}
	\lambda \mathbf{1} + y^{-1}w \in G(A) & \iff & \lambda y + w \in G(A) \\ 
	& \iff & \lambda \varphi(y) + \psi(w) \in G(B) \\
	& \iff & \lambda \mathbf{1} + \varphi(y)^{-1}\psi(w) \in G(B)
\end{eqnarray*}
for any $w \in A$ and $\lambda \in \mathbb{C}$. Consequently,
$$\sigma (y^{-1}w) = \sigma(\varphi(y)^{-1}\psi(w)) \mbox{ for all }w \in A.$$
Since $\sigma (y^{-1}x) = \sigma(\varphi(y)^{-1}\psi(x))$, it follows from the definition of $E(\psi(x))$ and Lemma~\ref{rank} that
$$\# \sigma'(y^{-1}x) = \# \sigma'(\varphi(y)^{-1}\psi(x)) = \R(\psi(x)) = \R(x).$$
Since $\R(x) = \R(y^{-1}x)$ and $\R(\psi(x)) = \R(\varphi(y)^{-1}\psi(x))$, we can now use \eqref{eqb} to conclude that
\begin{equation}
\tr(\varphi(y)^{-1}\psi(x)) = \tr(y^{-1}x).
\label{eqc}
\end{equation}
Moreover, since $\sigma(y^{-1}x_j) = \sigma(\varphi(y)^{-1}\psi(x_j))$ and $$\R(y^{-1}x_j) = \R(\varphi(y)^{-1}\psi(x_j)) = 1$$ 
for each $j \in \left\{1, \ldots, n\right\}$, it follows from \eqref{eqa} that
$$\tr(y^{-1}x_j) = \tr(\varphi(y)^{-1}\psi(x_j)) \mbox{ for each } j \in \left\{1, \ldots, n\right\}.$$
Thus, using the linearity of the trace, we have
\begin{eqnarray}
	\tr(y^{-1}x) & = & \tr(y^{-1}x_1) + \cdots + \tr(y^{-1}x_n) \nonumber \\ & = & \tr(\varphi(y)^{-1}\psi(x_1)) + \cdots + \tr(\varphi(y)^{-1}\psi(x_n)).
	\label{eqd}
\end{eqnarray}
Combining \eqref{eqc} with \eqref{eqd} now allows us to conclude that
\begin{equation}
\tr(u\psi(x)) = \tr(u\psi(x_1)) + \cdots + \tr(u\psi(x_n))
\label{eqe}
\end{equation}
for all $u \in E(\psi(x)) \cap G(B)$. Since the set $E(\psi(x))$ is dense in $B$, and the trace is continuous on the set of all elements with rank at most $\max\{\R(\psi(x)), 1\}$, we infer that \eqref{eqe} holds for all $u \in G(B)$. Now let $v \in B$ be given, and let 
$$b = \psi(x)-(\psi(x_1)+ \cdots + \psi(x_n)).$$
For all nonzero $\lambda \in \mathbb{C}$ with $1/|\lambda| > \rho(v)$ it now follows from \eqref{eqe} (which holds for all $u \in G(B)$) that $\tr((\mathbf{1}-\lambda v)^{-1}b) = 0$. Thus, by the continuity property and linearity of the trace, we have
$$0 = \tr\left(\sum_{j=0}^{\infty} \lambda^j v^jb\right) = \sum_{j=0}^{\infty} \lambda^j \tr(v^jb).$$
Consequently, $\tr(v^jb) = 0$ for all $j \in \mathbb{N}$. In particular, $\tr(bv)=\tr(vb) = 0$ for all $v \in \soc(A)$. Thus, since $b \in \soc(A)$  we conclude that $b = 0$ by Theorem~\ref{trprop}. Hence,
$$\psi(x_1 + \cdots + x_n) = \psi(x_1)+ \cdots + \psi(x_n),$$
and the additivity of $\psi$ on $\soc(A)$ follows. In a similar way one establishes the additivity of $\varphi$ on $\soc(A)$.\end{proof}

We are now ready to give a proof of our main result.

\begin{proof}[Proof of Theorem 3.1]
	From Lemma~\ref{inject} it follows that $\varphi, \psi: A \to B$ are bijective mappings. Moreover, $B$ is semisimple by Lemma~\ref{invert}. Notice now that $\varphi^{-1}, \psi^{-1}: B \to A$ are surjective mappings satisfying
		\begin{equation}
		\lambda a+b \in G(B) \iff \lambda \varphi^{-1}(a) + \psi^{-1}(b) \in G(A) \mbox{ for all }a, b \in B \mbox{ and }\lambda \in \mathbb{C},
		\label{eqf}
	\end{equation}
	but with the added advantage that the domain is semisimple and the codomain is semisimple with an essential socle. From Lemma~\ref{soc} it follows that the restrictions
	$\varphi^{-1}: \soc(B) \to \soc(A)$ and $\psi^{-1}: \soc(B) \to \soc(A)$ are well-defined  linear mappings. Since we already know that $\varphi^{-1}$ and $\psi^{-1}$ are homogeneous by Lemma~\ref{hom}, we now proceed to systematically extend the additivity of $\varphi^{-1}$ on $\soc(B)$ to all of $B$. In the process we also show that $\varphi^{-1} = \psi^{-1}$. The details are as follows.
	
	Let $u \in G(B)$ and $a \in \mathscr{F}_1(B)$ be given. Notice that $\varphi^{-1}(u)$ and $\varphi^{-1}(u+a)-\psi^{-1}(a)$ are both invertible in $A$ by Lemma~\ref{invert} and \eqref{eqf}. Moreover, for any $b \in \mathscr{F}_1(B)$, by using \eqref{eqf} and the linearity of $\psi^{-1}$ on $\soc(B)$ we have
	\begin{eqnarray*}
		& & \varphi^{-1}(u) + \psi^{-1}(b) \in G(A) \\ 
		& \iff & u + b \in G(B) \\
		& \iff & (u+a) - a+b \in G(B) \\
		& \iff & \varphi^{-1}(u+a) + \psi^{-1}(-a+b) \in G(A) \\
		& \iff &  \varphi^{-1}(u+a)  -\psi^{-1}(a) + \psi^{-1}(b) \in G(A). 
	\end{eqnarray*}
	Thus, since $\psi^{-1}(\mathscr{F}_1(B)) = \mathscr{F}_1(A)$ and $\soc(A)$ is essential, it follows from  Proposition~\ref{prop3} that
	$$\varphi^{-1}(u+a) = \varphi^{-1}(u) + \psi^{-1}(a) \mbox{ for each }u \in G(B) \mbox{ and } a \in \mathscr{F}_1(B).$$
	Now let $u, v \in G(B)$ be given. As before we have that $\varphi^{-1}(u)$ and $\psi^{-1}(u+v)-\varphi^{-1}(v)$ are both invertible in $A$.
	Moreover, for any $b \in \mathscr{F}_1(B)$, we can now use \eqref{eqf} and the homogeneity of $\varphi^{-1}$ to find that 
	\begin{eqnarray*}
		& & \psi^{-1}(u + v) - \varphi^{-1}(v) + \psi^{-1}(b) \in G(A) \\ 
		& \iff & \psi^{-1}(u + v) + \varphi^{-1}(-v + b) \in G(B) \\
		& \iff & (u+v) + (-v+b) \in G(B) \\
		& \iff & u + b \in G(A) \\
		& \iff &  \varphi^{-1}(u)  + \psi^{-1}(b) \in G(A). 
	\end{eqnarray*}
	Consequently, Proposition~\ref{prop3} yields that
	\begin{equation}
	\psi^{-1}(u+v) = \varphi^{-1}(u)+ \varphi^{-1}(v) \mbox{ for all }u, v \in G(B).
	\label{EQ1}
	\end{equation}
	If we interchange the roles of $\varphi^{-1}$ and $\psi^{-1}$ in both parts of the preceding argument, then we also obtain
	\begin{equation}
		\varphi^{-1}(u+v) = \psi^{-1}(u)+ \psi^{-1}(v) \mbox{ for all }u, v \in G(B).
		\label{EQ2}
	\end{equation}
	Taking $u = v$ in \eqref{EQ2} and using the homogeneity of $\varphi^{-1}$, we infer that
	\begin{equation}
	\varphi^{-1}(u) = \psi^{-1}(u) \mbox{ for all }u \in G(B).
	\label{EQ3}
	\end{equation}
	Now let $x \in B$ be arbitrary, and fix any $\lambda \in \mathbb{C}$ with $|\lambda| > \rho(x)$. Then $-\lambda \mathbf{1}, \lambda \mathbf{1} + x \in G(B)$. Thus, by \eqref{EQ1}, \eqref{EQ2}, and \eqref{EQ3}, we have
	\begin{eqnarray*}
		\psi^{-1}(x) & = & \psi^{-1}(-\lambda \mathbf{1} + \lambda\mathbf{1} + x) \\
		& = & \varphi^{-1}(-\lambda \mathbf{1}) + \varphi^{-1}(\lambda\mathbf{1} + x) \\
		& = & \psi^{-1}(-\lambda \mathbf{1}) + \psi^{-1}(\lambda\mathbf{1} + x) \\
		& = & \varphi^{-1}(-\lambda \mathbf{1} + \lambda\mathbf{1} + x) \\
		& = & \varphi^{-1}(x).
	\end{eqnarray*}
	Hence, $\varphi^{-1} = \psi^{-1}$. We are now in a position to show that $\varphi^{-1}$ is additive and hence linear. Let $v \in G(B)$ and $x \in B$ be arbitrary. Keeping in mind that $\varphi^{-1} = \psi^{-1}$ is an homogeneous map which satisfies \eqref{EQ1}, for any $u \in G(B)$, we have from \eqref{eqf} that 
	\begin{eqnarray*}
		& & \varphi^{-1}(x+v) - \varphi^{-1}(v) + \varphi^{-1}(u) \in G(A) \\ 
		& \iff & \varphi^{-1}(x+v) + \varphi^{-1}(-v + u) \in G(A) \\
		& \iff & x+v-v+u \in G(B) \\
		& \iff & x+u \in G(B) \\
		& \iff &  \varphi^{-1}(x) + \varphi^{-1}(u) \in G(A). 
	\end{eqnarray*}
	Hence, since $\varphi^{-1}(G(B)) = G(A)$ by Lemma~\ref{invert}, it follows from Theorem~\ref{thrm1} that
	$$\varphi^{-1}(x+v) = \varphi^{-1}(x) + \varphi^{-1}(v) \mbox{ for all } v \in G(B)\mbox{ and } x \in B.$$
	In light of the preceding identity and the homogeneity of $\varphi^{-1}$, to obtain the additivity of $\varphi^{-1}$, we can now simply repeat the preceding argument with $v \in G(B)$ replaced by $v \in B$. This then establishes the linearity of $\varphi^{-1} = \psi^{-1}$. Consequently, we have that $\varphi = \psi$ is linear.
	 
	 If we now let $J: A \to B$ be defined by $J(x) = \varphi(\mathbf{1})^{-1}\varphi(x)$ for each $x \in A$, then $J$ is a unital linear bijection which preserves invertibility. Thus, since $A$ and $B$ are semisimple and $\soc(A)$ is essential, it follows from \cite[Theorem 1.1]{bfs2003} that $J$ is a Jordan isomorphism. Hence, with $u = \varphi(\mathbf{1})$, we have
	$\varphi(x) = \psi(x) = uJ(x)$ for all  $x \in A$. 
\end{proof}

Using Sourour's results \cite[Theorem 1.1 and Corollary 1.2]{Sourour1996}, we obtain, as a consequence of Theorem~\ref{main}, the following classification for pairs of surjective mappings which preserve the invertibility of linear operator pencils in both directions. Here we use $X'$ to denote the dual space of $X$, and $S^{\ast}$ to denote the adjoint of an operator $S$.

\begin{cor}\label{operator}
		Let $\varphi, \psi: \mathcal{L}(X) \to \mathcal{L}(Y)$ be two surjective mappings which satisfy
	\begin{equation*}
		\lambda S+T \in G(\mathcal{L}(X)) \iff \lambda \varphi(S) + \psi(T) \in G(\mathcal{L}(Y)) \mbox{ for all }S, T \in \mathcal{L}(X) \mbox{ and }\lambda \in \mathbb{C}.
	\end{equation*}
	Then either
	$$\varphi(S) = \psi(S) = TSU \mbox{ for every } S \in \mathcal{L}(X)$$
	or
	$$\varphi(S) = \psi(S) = VS^{\ast}W \mbox{ for every } S \in \mathcal{L}(X),$$
	where $T: X \to Y$, $U: Y \to X$, $V: X' \to Y$ and $W: Y \to X'$ are bounded invertible operators.
\end{cor}

\providecommand{\bysame}{\leavevmode\hbox to3em{\hrulefill}\thinspace}

\end{document}